\documentclass[a4paper,11pt,reqno]{article}

\usepackage{amsmath,amssymb,amsfonts,amsthm}
\newtheorem{theorem}{Theorem}[section]
\newtheorem{maintheorem}{Theorem}

\newtheorem{proposition}[theorem]{Proposition}

\newtheorem{lemma}[theorem]{Lemma}
\theoremstyle{definition}
\newtheorem{definition}[theorem]{Definition}
\newtheorem{example}[theorem]{Example}
\newtheorem{remark}[theorem]{Remark}
\numberwithin{equation}{section}
\usepackage[dvips]{graphicx}
\usepackage[all]{xy}
\usepackage{color}

\newcommand{\dif}{\textsf{Diff}}

\newcommand{\sing}{\textsf{Sing}}
\newcommand{\dm}{\textsf{dim}}

\begin{document}
\centerline{\LARGE\bf On Levi-flat hypersurfaces tangent to}
\centerline{\LARGE\bf  holomorphic webs}
\vskip.3in
\centerline{\large Arturo Fern\'andez-P\'erez \footnote {Work supported by CNPq-Brazil.\\
 Keywords: Levi-flat hypersurfaces - Holomorphic webs\\
 Mathematical subjet classification: 32V40 - 14C21
}}
\vskip .2in

{\small\bf Abstract.} {\small\it We investigate germs of real analytic Levi-flat hypersurfaces tangent to germs of codimension one holomorphic webs. We introduce the notion of first integrals for local webs. In particular, we prove that a $k$-web with finitely many invariant analytic subvarieties through the origin tangent to a Levi-flat hypersurface has a holomorphic first integral.} 
\vskip .2in
{\small\bf R\'esum\'e.} {\small\it Nous \'etudions les germes d'hypersurfaces  r\'eelles analytiques Levi-plate tangente \`a les germes d'webs holomorphes codimension un. Nous introduisons la notion des int\'egrales premi\`eres des webs locales. En particulier, nous montrons que une $k$-web avec un nombre fini de feuilles invariant analytique par l'origine, tangente \`a une hypersurface Levi-plate poss\`ede une int\'egrale premi\`ere holomorphe.}

\section{Introduction}
\par In very general terms, a germ of codimension one $k$-web is a collection of $k$ germs of codimension one holomorphic foliations in ``general position". The study of webs was initiated by Blaschke and his school in the late 1920s. For a recent account of the theory, we refer the reader to \cite{pereira}. 
\par For instance, take $\omega\in Sym^{k}\Omega^{1}(\mathbb{C}^{2},0)$ defined by
 $$\omega=(dy)^{k}+a_{k-1}(dy)^{k-1}dx+\ldots+a_{0}(dx)^{k},$$
where $a_{j}\in\mathcal{O}_{2}$ for all $0\leq j\leq k-1$. Then $\mathcal{W}:\omega=0$, define a non-trivial $k$-web on $(\mathbb{C}^{2},0)$. In this paper we study webs and its relation with Levi-flat hypersurfaces. 
\par Let $M$ be a germ at $0\in\mathbb{C}^{n}$ of a real codimension one analytic irreducible analytic set. Since $M$ is real analytic
of codimension one, it can be decomposed into $M_{reg}$ and $\sing(M)$, where $M_{reg}$ is a germ of smooth real analytic hypersurface in $\mathbb{C}^{n}$ and $\sing(M)$, the singular locus, is contained in a proper analytic subvariety of lower dimension. 
 We shall say that $M$ is \textit{Levi-flat} if the complex distribution $L$ on $M_{reg}$  
\begin{align}
L_{p}:=T_{p}M\cap i T_{p}M\subset T_{p}M,\,\,\,\,\text{for any} \,\,p\in M_{reg}
\end{align}
is integrable, in Frobenius sense. It follows that $M_{reg}$ is smoothly foliated by immersed complex manifolds of complex dimension $n-1$. The foliation defined by $L$ is called the Levi foliation and will be denoted by $\mathcal{L}_{M}$. 
\par If $M$ is a real analytic smooth Levi-flat hypersurface, by a classic result of E. Cartan there exists a local holomorphic coordinates $(z_{1},\ldots,z_{n})\in\mathbb{C}^{n}$ such that $M$ can be represented by $M=\{\mathcal{I}m(z_{n})=0\}$. The situation if different if the hypersurface have singularities. 
Singular Levi-flat real analytic hypersurfaces have been studied by Burns and Gong \cite{burns}, Brunella \cite{brunella}, Lebl \cite{lebl}, the author \cite{normal}, \cite{arturo} and many others.
\par Recently D.Cerveau and A. Lins Neto \cite{alcides} have studied codimension one holomorphic foliations tangent to singular Levi-flat hypersurfaces.
A codimension one holomorphic foliation $\mathcal{F}$ is tangent to $M$, if any leaf of $\mathcal{L}_{M}$ is also a leaf of $\mathcal{F}$. In \cite{alcides} it is proved that a germ of codimension one holomorphic foliation tangent to a real analytic Levi-flat hypersurface has a non-constant meromorphic first integral. In the same spirit, the authors propose a problem for webs, which is as follows:

{\bf{Problem.-}} Let $M$ be a germ at $0\in\mathbb{C}^{n}$, $n\geq 2$, of real analytic hypersurface Levi-flat.
Assume that there exists a singular codimension one $k$-web, $k\geq 2$, such that any leaf of the Levi foliation $\mathcal{L}_{M}$ on $M_{reg}$ is also a leaf of the web. Does the web has a
non-constant meromorphic first integral?.
\par By a meromorphic first integral we mean something like $f_{0}(x)+z.f_{1}(x)+\ldots+
z^{k}.f_{k}(x)=0$, where $f_{0},f_{1},\ldots,f_{k}\in\mathcal{O}_{n}$. In this situation, the web is obtained by the elimination of $z$ in the system given by
\[ \left\{ \begin{array}{ll}
     f_{0}+z.f_{1}+z^{2}.f_{2}+\ldots+z^{k}.f_{k}=0     & \\
      df_{0}+z.df_{1}+z^{2}.df_{2}+\ldots+z^{k}.df_{k}=0. & 
\end{array} \right. \]
\par In this work, we organize some results on singular Levi-flat hypersurfaces and holomorphic foliations which provide a best approach to study of webs and Levi-flats. Concerning the problem, we obtain an interesting result in a case very special (Theorem \ref{main-the}), the problem remains open in general.
\subsection{Local singular webs}\label{definition-web}
\par It is customary to define a germ of singular holomorphic foliation as an equivalence class $[\omega]$ of germs of holomorphic 1-forms in $\Omega^{1}(\mathbb{C}^{n},0)$ modulo multiplication by elements of $\mathcal{O}^{*}(\mathbb{C}^{n},0)$ such that any representative $\omega$ is integrable ( $\omega\wedge d\omega=0$ ) and with singular set $\sing(\omega)=\{p\in(\mathbb{C}^{n},0):\omega(p)=0\}$ of codimension at least two. 
\par An analogous definition can be made for codimension one $k$-webs.
 A germ at $(\mathbb{C}^{n},0)$, $n\geq 2$ of codimension one $k$-web $\mathcal{W}$ is an equivalence class $[\omega]$ of germs 
of $k$-symmetric 1-forms, that is sections of $Sym^{k}\Omega^{1}(\mathbb{C}^{n},0)$, modulo multiplication by $\mathcal{O}^{*}(\mathbb{C}^{n},0)$ such that a suitable representative $\omega$ defined in a connected neighborhood $U$ of the origin satisfies the following conditions:
\begin{enumerate}
\item The zero set of $\omega$ has codimension at least two. 
\item The 1-form $\omega$, seen as a homogeneous polynomial of degree $k$ in the ring $\mathcal{O}_{n}[dx_{1},\ldots,dx_{n}]$, is square-free.
\item (Brill's condition) For a generic $p\in U$, $\omega(p)$ is a product of $k$ linear forms.
\item (Frobenius's condition) For a generic $p\in U$, the germ of $\omega$ at $p$ is the product of $k$ germs of integrable 1-forms.
\end{enumerate}
\par Both conditions $(3)$ and $(4)$ are automatic for germs at $(\mathbb{C}^{2},0)$ of webs and non-trivial for germs at $(\mathbb{C}^{n},0)$ when $n\geq 3$.
\par We can think $k$-webs as first order differential equations
of degree $k$. The idea is to consider the germ of web as a meromorphic section of the projectivization of the cotangent bundle of $(\mathbb{C}^{n},0)$. This is a classical point view in the theory of differential equations, which has been recently explored in Web-geometry. For instance see \cite{cavalier}, \cite{lehmann}, \cite{yartey}. 
\subsection{The contact distribution}
\par Let us denote $\mathbb{P}:=\mathbb{P}T^{*}(\mathbb{C}^{n},0)$ the projectivization of the cotangent bundle of $(\mathbb{C}^{n},0)$
and $\pi:\mathbb{P}T^{*}(\mathbb{C}^{n},0)\rightarrow(\mathbb{C}^{n},0)$ the natural projection. Over a point $p$ the fiber $\pi^{-1}(p)$ parametrizes the one-dimensional subspaces of $T^{*}_{p}(\mathbb{C}^{n},0)$. On $\mathbb{P}$  there is a canonical codimension one distribution, the so called contact distribution $\mathcal{D}$. Its description in terms of a system of coordinates $x=(x_{1},\ldots,x_{n})$ of $(\mathbb{C}^{n},0)$ goes as follows: let $dx_{1},\ldots,dx_{n}$ be the basis of $T^{*}(\mathbb{C}^{n},0)$ associated to the coordinate system $(x_{1},\ldots,x_{n})$. Given a point $(x,y)\in T^{*}(\mathbb{C}^{n},0)$, we can write $y=\sum_{j=1}^{n}y_{j}dx_{j}$, $(y_{1},\ldots,y_{n})\in\mathbb{C}^{n}$. In this way, if $(y_{1},\ldots,y_{n})\neq 0$ then we set $[y]=[y_{1},\ldots.y_{n}]\in\mathbb{P}^{n-1}$ and $(x,[y])\in(\mathbb{C}^{n},0)\times\mathbb{P}^{n-1}\cong\mathbb{P}$. In the affine coordinate system $y_{n}\neq 0$ of $\mathbb{P}$, the distribution $\mathcal{D}$ is defined by $\alpha=0$, where
\begin{align}
\alpha=dx_{n}-\sum_{j=1}^{n-1}p_{j}dx_{j},\,\,\,\,\,\,\,\, p_{j}=-\frac{y_{j}}{y_{n}}\,\,\,\,\,\,\,\,\,(1\leq j\leq n-1).
\end{align}
The 1-form $\alpha$ is called the contact form.

\subsection{Webs as closures of meromorphic multi-sections}
\par Let us consider $X\subset\mathbb{P}$ a subvariety, not necessarily irreducible, but of pure dimension $n$. Let $\pi_{X}:X\rightarrow (\mathbb{C}^{n},0)$ be the restriction to $X$ of the projection $\pi$. Suppose also that $X$ satisfies the following conditions:
\begin{enumerate}
 \item The image under $\pi$ of every irreducible component of $X$ has dimension $n$.
\item The generic fiber of $\pi$ intersects $X$ in $k$ distinct smooth points and at these the differential $d\pi_{X}:T_{p}X\rightarrow T_{\pi(p)}(\mathbb{C}^{n},0)$ is surjective. Note that $k=\deg(\pi_{X})$.
\item The restriction of the contact form $\alpha$ to the smooth part of every irreducible component of $X$ is integrable. We denote $\mathcal{F}_{X}$ the foliation defined by $\alpha|_{X}=0$.
\end{enumerate}
\par We can define a germ $\mathcal{W}$ at $0\in\mathbb{C}^{n}$ of $k$-web as a triple $(X,\pi_{X},\mathcal{F}_{X})$. This definition is equivalent to  one given in Section \ref{definition-web}. In the sequel, $X$ will always be the variety associated to $\mathcal{W}$, the singular set of $X$ will be denoted by $\sing(X)$ and its the smooth part will be denoted by $X_{reg}$.
\begin{definition}
 Let $R$ be the set of points $p\in X$ where
\begin{itemize}
 \item either $X$ is singular, 
\item or the differential $d\pi_{X}:T_{p}X_{reg}\rightarrow T_{\pi(p)}(\mathbb{C}^{n},0)$ is not an isomorphism.
\end{itemize}
 \end{definition}
 The analytic set $R$ is called the criminant set of $\mathcal{W}$ and $\Delta_{\mathcal{W}}=\pi(R)$ the discriminant of $\mathcal{W}$. Note that $\dm(R)\leq n-1$.
\begin{remark}
Let $\omega\in Sym^{k}\Omega_{1}(\mathbb{C}^{n},0)$ and assume that it defines a $k$-web $\mathcal{W}$ with variety $X$. Then $X$ is irreducible if, and only if, $\omega$ is irreducible in the ring $\mathcal{O}_{n}[dx_{1},\ldots,dx_{n}]$. In this case we say that the web is irreducible.
\end{remark}
\par  Let $M$ be a germ at $0\in\mathbb{C}^{n}$ of a real analytic Levi-flat hypersurface.  
\begin{definition}\label{web-tang}
 We say that $M$ is tangent to $\mathcal{W}$ if any leaf of the Levi foliation $\mathcal{L}_{M}$ on $M_{reg}$ is also a leaf of $\mathcal{W}$.
\end{definition} 
\subsection{First integrals for webs}
 
\begin{definition}
We say that $\mathcal{W}$ a $k$-web has a meromorphic first integral if, and only if, there exists $$P(z)=f_{0}+z.f_{1}+\ldots+z^{k}.f_{k}\in\mathcal{O}_{n}[z],$$ where
$f_{0},\ldots,f_{k}\in\mathcal{O}_{n}$, such that every irreducible component of the hypersurface $(P(z_{0})=0)$
is a leaf of $\mathcal{W}$, for all $z_{0}\in(\mathbb{C},0)$. 
\end{definition}
\begin{definition}\label{def-inte}
We say that $\mathcal{W}$ a $k$-web has a holomorphic first integral if, and only if, there exists $$P(z)=f_{0}+z.f_{1}+\ldots+ z^{k-1}.f_{k-1}+z^{k}\in\mathcal{O}_{n}[z],$$ where
$f_{0},\ldots,f_{k-1}\in\mathcal{O}_{n}$, such that every irreducible component of the hypersurface $(P(z_{0})=0)$
is a leaf of $\mathcal{W}$, for all $z_{0}\in(\mathbb{C},0)$. 
\end{definition}
\par We will prove a result concerning the situation of definitions \ref{web-tang} and \ref{def-inte}. 

\begin{maintheorem}\label{main-the}
Let $\mathcal{W}$ be a germ at $0\in\mathbb{C}^{n}$, $n\geq{2}$ of $k$-web defined by 
$$\omega=\underset{i_{1},\ldots,i_{n}\geq 0}{ \sum_{i_{1}+\ldots+i_{n}=k}} a_{i_{1},\ldots, i_{n}}(z)dz_{1}^{i_{1}}\ldots dz_{n}^{i_{n}},$$
where $a_{i_{1},\ldots, i_{n}}\in\mathcal{O}_{n}$ and $a_{0,0,\ldots,0,k}(0)\neq 0$. Suppose that $\mathcal{W}$ is 
tangent to a germ  at $0\in\mathbb{C}^{n}$ of an irreducible real-analytic Levi-flat hypersurface $M$. Furthermore,  
assume that $\mathcal{W}$ is irreducible and has finitely many invariant analytic subvarieties through the origin. Let $X$ be the variety associated to $\mathcal{W}$. Then $\mathcal{W}$ has a non-constant holomorphic first integral, if one of the following conditions is fulfilled :
\begin{enumerate}
 \item If $n=2$.
\item If $n\geq 3$ and $cod_{X_{reg}}(\sing(X))\geq 2$. 
\end{enumerate}
Moreover, if $P(z)=f_{0}+z.f_{1}+\ldots+ z^{k-1}.f_{k-1}+z^{k}\in\mathcal{O}_{n}[z]$ is a holomorphic first integral for $\mathcal{W}$, then $M=(F=0)$, where $F$ is obtained by the elimination of $z$ in the system given by
\[ \left\{ \begin{array}{ll}
     f_{0}+z.f_{1}+z^{2}.f_{2}+\ldots+z^{k-1}.f_{k-1}+z^{k}=0     & \\
      \bar{f}_{0}+z.\bar{f}_{1}+z^{2}.\bar{f}_{2}+\ldots+z^{k-1}.\bar{f}_{k-1}+z^{k}=0. & 
\end{array} \right. \]

\end{maintheorem}

\begin{remark} 
Under the hypotheses of Theorem \ref{main-the}, if $n=2$ and $k=1$, $\mathcal{W}$ is a non-dicritical holomorphic foliation at $(\mathbb{C}^{2},0)$ tangent to a germ of an irreducible real analytic Levi-flat hypersurface $M$, then Theorem 1 given by Cerveau and Lins Neto \cite{alcides} assures that $\mathcal{W}$ has a non-constant holomorphic first integral. In this sense, our theorem is a generalization of result of Cerveau and Lins Neto.
\end{remark} 
\begin{remark}
 Let $\mathcal{W}$ a germ at $0\in\mathbb{C}^{n}$, $n\geq{2}$, of a smooth $k$-web tangent to a germ at $0\in\mathbb{C}^{n}$ of an irreducible real codimension one submanifold $M$. In other words, $\mathcal{W}=\mathcal{F}_{1}\boxtimes\ldots\boxtimes\mathcal{F}_{k}$ is a generic superposition of $k$ germs at $0\in\mathbb{C}^{n}$ of smooth foliations $\mathcal{F}_{1},\ldots,\mathcal{F}_{k}$. In this case the irreducibility and tangency conditions to $M$ implies the existence of a unique $i\in \{1,\ldots,k\}$ such that $\mathcal{F}_{i}$ is tangent to $M$. Therefore we can find a coordinates system $z_{1},\ldots,z_{n}$ of $\mathbb{C}^{n}$ such that $\mathcal{F}_{i}$ is defined by $dz_{n}=0$ and $M=(\mathcal{I}m(z_{n})=0)$.  
\end{remark}

\section{The foliation associated to a web}
 In this section, we prove a key lemma which will be used in the proof of main theorem.
\par Since the restriction of $\mathcal{D}$  to $X_{reg}$ is integrable, it defines a foliation $\mathcal{F}_{X}$, which in general is a singular foliation. Given $p\in(\mathbb{C}^{n},0)\backslash\Delta_{\mathcal{W}}$, $\pi_{X}^{-1}(p)=\{q_{1},\ldots,q_{k}\}$, where $q_{i}\neq q_{j}$, if $i\neq j$,
($\deg(\pi_{X})=k$), denote by $\mathcal{F}_{X}^{i}$ the germ of $\mathcal{F}_{X}$ at $q_{i}$, $i=1,\ldots,k$.
\par The projections $\pi_{*}(\mathcal{F}_{X}^{i}):=\mathcal{F}_{p}^{i}$ define $k$ germs of codimension one foliations at $p$.
\begin{definition}
A leaf of the web $\mathcal{W}$ is, by definition, the projection on $(\mathbb{C}^{n},0)$ of a leaf of $\mathcal{F}_{X}$.
\end{definition}
\begin{remark}
Given $p\in(\mathbb{C}^{n},0)\backslash\Delta_{\mathcal{W}}$, and $q_{i}\in\pi_{X}^{-1}(p)$, the projection $\pi_{X}(L_{i})$ of the leaf $L_{i}$ of $\mathcal{F}_{X}$ through $q_{i}$, gives rise to a leaf of $\mathcal{W}$ through $p$. In particular, $\mathcal{W}$ has at most $k$ leaves through $p$. 
\end{remark}
\par We will use the following proposition (cf. \cite{gunning} Th. 5, pg. 32). Let $\mathcal{O}(X)$ denote the ring of holomorphic functions on $X$.

\begin{proposition}\label{prop-hof}
Let $V$ be an analytic variety. If $\pi:V\rightarrow W$ is a finite branched holomorphic covering of pure order $k$ over an open subset $W\subseteq\mathbb{C}^{n}$, then to each holomorphic function $f\in\mathcal{O}(V)$ there is a canonically associated monic polynomial $P_{f}(z)\in\mathcal{O}_{n}[z]\subseteq\mathcal{O}(V)[z]$ of degree $k$ such that $P_{f}(f)=0$ in $\mathcal{O}(V)$. 
\end{proposition}
We have now the following lemma. 
\begin{lemma}\label{web-integral}
Suppose that $(X,\pi_{X},\mathcal{F}_{X})$ defines a $k$-web $\mathcal{W}$ on $(\mathbb{C}^{n},0)$, $n\geq 2$, where $X$ is an irreducible subvariety of $\mathbb{P}$. If $\mathcal{F}_{X}$ has a non-constant holomorphic first integral then $\mathcal{W}$ also has a holomorphic first integral. 
\end{lemma}
\begin{proof}
Let $g\in\mathcal{O}(X)$ be the first integral for $\mathcal{F}_{X}$. By Proposition \ref{prop-hof}, there exists a monic polynomial $P_{g}(z)\in\mathcal{O}_{n}[z]$ of degree $k$ such that $P_{g}(g)=0$ in $\mathcal{O}(X)$. 
Write  $$P_{g}(z)=g_{0}+z.g_{1}+\ldots+z^{k-1}.g_{k-1}+z^{k},$$ where $g_{0},\ldots,g_{k-1}\in\mathcal{O}_{n}$. 

\textbf{Assertion}.-- $P_{g}$ define a holomorphic first integral for $\mathcal{W}$.

\par Let $U\subseteq(\mathbb{C}^{n},0)\backslash\Delta_{\mathcal{W}}$ be an open subset and let $\varphi:X\rightarrow(\mathbb{C}^{n},0)\times\mathbb{C}$ be defined by $\varphi=(\pi_{X},g)$. Take a leaf $L$ of $\mathcal{W}|_{U}$.
Then there is $z\in\mathbb{C}$ such that the following diagram 
$$\xymatrix{ 
\ar[dr]_{\pi_{X}} \pi^{-1}_{X}(U)\cap\varphi^{-1}(L\times \{z\}) \ar[rr]^{\varphi} &  & L\times \{z\} \ar[dl]^{pr_{1}}\\
             & L
} $$
is commutative, where $pr_{1}$ is the projection on the first coordinate. It follows that $L$ is a leaf of $\mathcal{W}$ if and only if $g$ is constant along of each connected component of $\pi_{X}^{-1}(L)$ contained in $\varphi^{-1}(L\times \{z\})$.

\par Consider now the hypersurface $G=\varphi(X)\subset (\mathbb{C}^{n},0)\times\mathbb{C}$ which is the closure of set
$$\{(x,s)\in U\times\mathbb{C}: g_{0}(x)+s.g_{1}(x)+\ldots+s^{k-1}.g_{k-1}(x)+s^{k}=0\}.$$
\par Let $\psi:(\mathbb{C}^{n},0)\times \mathbb{C}\rightarrow (\mathbb{C}^{n},0)$ be the usual projection and denote by $Z\subset(\mathbb{C}^{n},0)$ the analytic subset such that the restriction to $G$ of $\psi$ not is a finite branched covering. Notice that for all $x_{0}\in(\mathbb{C}^{n},0)\backslash Z$,
the equation 
$$g_{0}(x)+s.g_{1}(x)+\ldots+s^{k-1}.g_{k-1}(x)+s^{k}=0$$
defines $k$ analytic hypersurfaces pairwise transverse in $x_{0}$ and therefore correspond to leaves of $\mathcal{W}$.  \end{proof}

\section{Examples}
\par This section is devoted to give some examples of Levi-flat hypersurfaces tangent to holomorphic foliations or webs. 
\begin{example}
Take a non constant holomorphic function $f:(\mathbb{C}^{n},0)\rightarrow(\mathbb{C},0)$ and set $M=(\mathcal{I}m(f)=0)$. Then $M$ is Levi-flat and $M_{sing}$ is the set of critical points of $f$ lying on $M$. Leaves of the Levi foliation on $M_{reg}$ are given by $\{f=c\}$, $c\in\mathbb{R}$. Of course, $M$ is tangent to a singular holomorphic foliation generated by the kernel of $df$.
\end{example}
\begin{example}\label{example-alcides}(\cite{alcides})
 Let $f_{0},f_{1},\ldots,f_{k}\in\mathcal{O}_{n}$, $n\geq{2}$, be irreducible germs of holomorphic
functions, where $k\geq{2}$. Consider the family of hypersurfaces
                    $$G:=\{G_{s}:= f_{0}+sf_{1}+\ldots+s^{k}f_{k}/ s\in\mathbb{R}\}.$$ 
By eliminating the real variable $s$ in the system $G_{s}=\bar{G}_{s}=0$, we obtain a real
analytic germ $F:(\mathbb{C}^{n},0)\rightarrow(\mathbb{R},0)$ such that any complex hypersurface $(G_{s}=0)$
is contained in the real hypersurface $(F = 0)$. For instance, in the case $k=2$, we
obtain

 \[F=det\left( \begin{array}{cccc}
f_{0} & f_{1} & f_{2} & 0 \\
0 & f_{0} & f_{1} & f_{2} \\
\bar{f}_{0} & \bar{f}_{1} & \bar{f}_{2} & 0\\
0 & \bar{f}_{0} & \bar{f}_{1} & \bar{f}_{2}
\end{array} \right)=\]
\begin{equation}\label{example-web}
=f^{2}_{0}.\bar{f}^{2}_{2}+\bar{f}^{2}_{0}.f^{2}_{2}+f_{0}.f_{2}.\bar{f}^{2}_{1}+\bar{f}_{0}.\bar{f}_{2}.f_{1}^{2}-
|f_{1}|^{2}(f_{0}.\bar{f}_{2}+\bar{f}_{0}.f_{2})-2|f_{0}|^{2}.|f_{2}|^{2}.
\end{equation}
which comes from the elimination of $s$ in the system
$$f_{0}+s.f_{1}+s^{2}.f_{2}=\bar{f}_{0}+s.\bar{f}_{1}+s^{2}.\bar{f}_{2}=0.$$
\par  We would like to observe that the examples of this type are tangent to singular
webs. The web is obtained by the elimination of $s$ in the system given by
\[ \left\{ \begin{array}{ll}
     f_{0}+s.f_{1}+s^{2}.f_{2}+\ldots+s^{k}.f_{k}=0     & \\
      df_{0}+s.df_{1}+s^{2}.df_{2}\ldots+s^{k}.df_{k}=0 & 
\end{array} \right. \] 
In the case we get a $2$-web given by the implicit differential equation $\Omega = 0$, where 
\[\Omega=det\left( \begin{array}{cccc}
f_{0} & f_{1} & f_{2} & 0 \\
0 & f_{0} & f_{1} & f_{2} \\
df_{0} & df_{1} & df_{2} & 0\\
0 & df_{0} & df_{1} & df_{2}
\end{array} \right)\]
\end{example}
 \par This example shows that, although $\mathcal{L}_{M}$ is a foliation on $M_{reg}\subset M=(F=
0)$, in general it is not tangent to a germ of holomorphic foliation at $(\mathbb{C}^{n},0)$.
\begin{example}\label{cla-equa}[Clairaut's equations]
 Clairaut's equations are tangent to Levi-flat hypersurfaces. Consider the first-order implicit differential equation
\begin{equation}\label{clai-web}
y=xp+f(p), 
\end{equation}
where $(x,y)\in\mathbb{C}^{2}$, $p=\frac{dy}{dx}$ and $f\in\mathbb{C}[p]$ is a polynomial of degree $k$, the equation (\ref{clai-web}) define a $k$-web $\mathcal{W}$ on $(\mathbb{C}^{2},0)$. The variety $S$ associated to $\mathcal{W}$ is given by $(y-xp-f(p)=0)$ and the foliation $\mathcal{F}_{S}$ is defined by $\alpha|_{S}=0$, where $\alpha=dy-pdx$.
In the chart $(x,p)$ of $S$, we get $\alpha|_{S}=(x+f'(p))dp$. The criminant set of $\mathcal{W}$ is given by $$R=(y-xp-f(p)=x+f'(p)=0).$$ 
\par Observe that $\mathcal{F}_{S}$ is tangent to $S$ along $R$ and has a non-constant first integral $g(x,p)=p$. Denote by $\pi_{S}:S\rightarrow (\mathbb{C}^{2},0)$ the restriction to $S$ of the usual projection $\pi:\mathbb{P}\rightarrow (\mathbb{C}^{2},0)$, then the leaves of $\mathcal{F}_{S}$ project by $\pi_{S}$ in leaves of $\mathcal{W}$. Those leaves are as follows
\begin{align}
-y+s.x+f(s)=0,  
\end{align}
where $s$ is a constant. By the elimination of the variable $s\in\mathbb{R}$ in the system
\[ \left\{ \begin{array}{ll}
     -y+s.x+f(s)=0     & \\
      -\bar{y}+s.\bar{x}+\overline{f(s)}=0, & 
\end{array} \right. \] 
we obtain a Levi-flat hypersurface tangent to $\mathcal{W}$. In particular, Clairaut's equation has a holomorphic first integral.
\end{example}

\section{Lifting of Levi-flat hypersurfaces to the cotangent bundle}
In this section we give some remarks about the lifting of a Levi-flat hypersurface to the cotangent bundle of $(\mathbb{C}^{n},0)$. 
\par Let $\mathbb{P}$ be as before, the projectivized cotangent bundle of $(\mathbb{C}^{n},0)$ and $M$ an irreducible real analytic Levi-flat at $(\mathbb{C}^{n},0)$, $n\geq 2$. Note that $\mathbb{P}$ is a $\mathbb{P}^{n-1}$-bundle over $(\mathbb{C}^{n},0)$, whose fiber $\mathbb{P}T_{z}^{*}\mathbb{C}^{n}$ over $z\in\mathbb{C}^{n}$ will be thought of as the set of complex hyperplanes in $T_{z}^{*}\mathbb{C}^{n}$. Let $\pi:\mathbb{P}\rightarrow(\mathbb{C}^{n},0)$ be the usual projection.
\par The regular part $M_{reg}$ of $M$ can be lifted to $\mathbb{P}$: just take, for every
$z\in M_{reg}$, the complex hyperplane 
\begin{equation}
T_{z}^{\mathbb{C}}M_{reg}=T_{z}M_{reg}\cap i(T_{z}M_{reg})\subset T_{z}\mathbb{C}^{n}. 
\end{equation}
We call 
\begin{equation}
M'_{reg}\subset\mathbb{P} 
\end{equation}
this lifting of $M_{reg}$. We remark that it is no more a hypersurface: its (real) dimension
$2n-1$ is half of the real dimension of $\mathbb{P}T^{*}\mathbb{C}^{n}$. However, it is still ``Levi-flat'', in a
sense which will be precised below.

\par Take now a point $y$ in the closure $\overline{M'_{reg}}$ projecting on $\mathbb{C}^{n}$ to a point $x\in \overline{M}$. Now, we shall consider the following results, which are adapted from \cite{brunella}.
\begin{lemma}\label{brunella-lemma}
 There exist, in a germ of neighborhood $U_{y}\subset \mathbb{P}T^{*}(\mathbb{C}^{n},0)$ of $y$, a germ of real analytic subset $N_{y}$
of dimension $2n-1$ containing $M'_{reg}\cap U_{y}$.
\end{lemma}

\begin{proposition}\label{brunella-proposition}
 Under the above conditions, in a germ of neighborhood $V_{y}\subset U_{y}$ of $y$, there exists a germ of complex analytic subset $Y_{y}$ of (complex) dimension $n$ containing $N_{y}\cap V_{y}$.
\end{proposition}

\section{Proof of Theorem \ref{main-the}}
 The proof will be divided in two parts. First, we give the proof for $n=2$. The proof in dimension $n\geq 3$ will be done by reduction to the case of dimension two. 
\par First of all, we recall some results (cf. \cite{alcides}) about foliations and Levi-flats. Let $M$ and $\mathcal{F}$ be germs at $(\mathbb{C}^{2},0)$ of a real analytic Levi-flat hypersurface and of a holomorphic foliation, respectively, where $\mathcal{F}$ is tangent to $M$. Assume that:
\begin{enumerate}
 \item[(i)] $\mathcal{F}$ is defined by a germ at $0\in\mathbb{C}^{2}$ of holomorphic vector field $X$ with an isolated singularity at $0$.
\item[(ii)] $M$ is irreducible.
\end{enumerate}
Let us assume that $0$ is a reduced singularity of $X$, in the sense of Seidenberg \cite{seidenberg}. Denote the eigenvalues of $DX(0)$ by $\lambda_{1},\lambda_{2}$.
\begin{proposition}\label{proposition-reducido}
 Suppose that $X$ has a reduced singularity at $0\in\mathbb{C}^{2}$ and is tangent to a real analytic Levi-flat hypersurface $M$. Then $\lambda_{1},\lambda_{2}\neq 0$, $\lambda_{2}/\lambda_{1}\in\mathbb{Q}_{-}$ and $X$ has a holomorphic first integral. \\ In particular, in a suitable coordinates system $(x,y)$ around $0\in\mathbb{C}^{2}$, $X=\phi.Y$, where $\phi(0)\neq 0$ and 
\begin{align}
Y=q.x\partial_{x}-p.y\partial_{y}\,\,, g.c.d(p,q)=1.
\end{align}
In this coordinate system, $f(x,y):=x^{p}.y^{q}$ is a first integral of $X$. 
\end{proposition}
\par We call this type of singularity of $\mathcal{F}$ a saddle with first integral, (cf. \cite{Loray}, pg. 162). Now we have the following lemma.
\begin{lemma}\label{cerveau-lema}
 For any  $z_{0}\in M_{reg}$, the leaf $L_{z_{0}}$ of $\mathcal{L}_{M}$ through $z_{0}$ is closed in $M_{reg}$.
\end{lemma}
\subsection{Planar webs} 
\par A $k$-web $\mathcal{W}$ on $(\mathbb{C}^{2},0)$ can be written in coordinates $(x,y)\in\mathbb{C}^{2}$ by
 $$\omega=a_{0}(x,y)(dy)^{k}+a_{1}(x,y)(dy)^{k-1}(dx)+\ldots+a_{k}(x,y)(dx)^{k}=0,$$
where the coefficients $a_{j}\in\mathcal{O}_{2}$, $j=1,\ldots,k$. We set $$U=\{(x,y,[adx+bdy])\in\mathbb{P}T^{*}(\mathbb{C}^{2},0):a\neq 0\}$$ and $$V=\{(x,y,[adx+bdy])\in\mathbb{P}T^{*}(\mathbb{C}^{2},0):b\neq 0\}.$$ 
 Note that $\mathbb{P}T^{*}(\mathbb{C}^{2},0)=U\cup V$. Suppose that $(S,\pi_{S},\mathcal{F}_{S})$ define $\mathcal{W}$,  in the coordinates $(x,y,p)\in U$, where $p=\frac{dy}{dx}$, we have 
\begin{enumerate}
 \item $S\cap U=\{(x,y,p)\in\mathbb{P}T^{*}(\mathbb{C}^{2},0):F(x,y,p)=0\},$
where $$F(x,y,p)=a_{0}(x,y)p^{k}+a_{1}(x,y)p^{k-1}+\ldots+a_{k}(x,y).$$
 Note that $S$ is possibly singular at $0$.
\item $\mathcal{F}_{S}$ is defined by $\alpha|_{S}=0$, where $\alpha=dy-pdx$.
\item The criminant set $R$ is defined by the equations $$F(x,y,p)=F_{p}(x,y,p)=0.$$
\end{enumerate}
\par In $V$ the coordinate system is $(x,y,q)\in\mathbb{C}^{3}$, where $q=\frac{1}{p}$, the equations are similar.

\subsection{Proof in dimension two}

\par Let $\mathcal{W}$ be a $k$-web tangent to $M$ Levi-flat and let us consider $S$, $\pi$ be as before. 
The idea is to use Lemma \ref{web-integral}, assume that $\mathcal{W}$ is defined by 
\begin{equation}\label{web-equation}
\omega=a_{0}(x,y)(dy)^{k}+a_{1}(x,y)(dy)^{k-1}dx+\ldots+a_{k}(x,y)(dx)^{k}=0,
\end{equation}
where the coefficients $a_{j}\in\mathcal{O}_{2}$, $j=1,\ldots,k$ and $a_{0}(0,0)=1$.

\begin{lemma}\label{lemma-intersection}
 Under the hypotheses of Theorem \ref{main-the} and above conditions, the surface $S$ is irreducible and $S\cap\pi^{-1}(0)$ contains just a number finite of points. See figure 1.
\end{lemma}
\begin{proof}
Since $\mathcal{W}$ is irreducible so is $S$. On the other hand, $S\cap\pi^{-1}(0)$ is finite because 
$\mathcal{W}$  has a finite number of invariant analytic leaves through the origin and is defined as in \ref{web-equation}.
\end{proof}

\begin{figure}
\begin{center}
\includegraphics[scale=0.85]{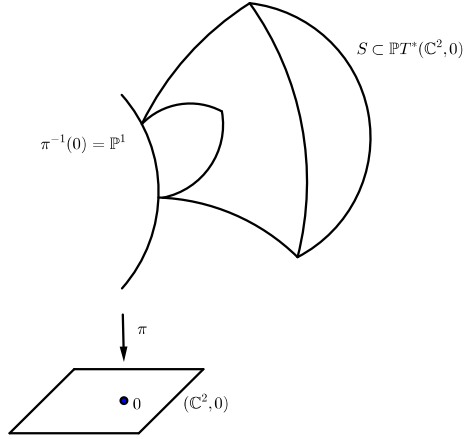}
\end{center}
\caption{$S\cap\pi^{-1}(0)$}
\end{figure}

\par We can assume without lost of generality that $S\cap\pi^{-1}(0)$ contains just one point, in the case general, 
the idea of the proof is the same. Then in the coordinate system $(x,y,p)\in\mathbb{C}^{3}$, where $p=\frac{dy}{dx}$, we have $\pi^{-1}(0)\cap S=\{p_{0}=(0,0,0)\}$, which implies that $S$ must be singular at $p_{0}\in\mathbb{P}T^{*}(\mathbb{C}^{2},0)$. In particular, $(S,p_{0})$ the germ of $S$ at $p_{0}$ is
defined by $F^{-1}(0)$, where $$F(x,y,p)=p^{k}+a_{1}(x,y)p^{k-1}+\ldots+a_{k}(x,y),$$
and $a_{1},\ldots,a_{k}\in\mathcal{O}_{2}$. Let $\mathcal{F}_{S}$ be the foliation defined by $\alpha|_{S}=0$. 
The assumptions implies that $\mathcal{F}_{S}$ is a non-dicritical foliation with an isolated singularity at $p_{0}$. 
\par Recall that a germ of foliation $\mathcal{F}$ at $p_{0}\in S$ is dicritical if it has infinitely many
analytic separatrices through $p_{0}$. Otherwise it is called non-dicritical.
\par Let $M'_{reg}$ be the lifting of $M_{reg}$ by $\pi_{S}$, and denote by $\sigma:(\tilde{S},D)\rightarrow(S,p_{0})$ the resolution of singularities of $S$ at $p_{0}$. Let $\tilde{\mathcal{F}}=\sigma^{*}(\mathcal{F}_{S})$ be the pull-back of 
$\mathcal{F}_{S}$ under $\sigma$. See figure 2.
\begin{lemma}\label{sing-int}
 In the above situation. The foliation 
$\tilde{\mathcal{F}}$ has only singularities of saddle with first integral type in $D$.
\end{lemma}
\begin{proof}
 Let $y\in\overline{M'_{reg}}$, it follows from Lemma \ref{brunella-lemma} the existence, in a neighborhood $U_{y}\subset \mathbb{P}T^{*}(\mathbb{C}^{2},0)$ containing $y$, of a real analytic subset $N_{y}$ of dimension $3$ containing $M'_{reg}\cap U_{y}$. Then by Proposition \ref{brunella-proposition}, there exists, in a neighborhood $V_{y}\subset U_{y}$ of $y$, a complex analytic subset $Y_{y}$ of (complex) dimension $2$ containing $N_{y}\cap V_{y}$. As germs at $y$, we get $Y_{y}=S_{y}$ then $N_{y}\cap V_{y}\subset S_{y}$, we have that $N_{y}\cap V_{y}$ is a real analytic hypersurface in $S_{y}$, and it is Levi-flat because each irreducible component contains a Levi-flat piece (cf. \cite{burns}, Lemma 2.2).
\par Let us denote $M'_{y}=N_{y}\cap V_{y}$. The hypotheses implies that $\mathcal{F}_{S}$ is tangent to $M'_{y}$. These local constructions are sufficiently 
canonical to be patched together, when $y$ varies on $\overline{M'_{reg}}$: if $S_{y_{1}}\subset V_{y_{1}}$ and $S_{y_{2}}\subset V_{y_{2}}$ are as above,
 with $M'_{reg}\cap V_{y_{1}}\cap V_{y_{2}}\neq \emptyset$, then $S_{y_{2}}\cap( V_{y_{1}}\cap V_{y_{2}})$ and $S_{y_{1}}\cap( V_{y_{1}}\cap V_{y_{2}})$ 
have some common irreducible components containing $M'_{reg}\cap V_{y_{1}}\cap V_{y_{2}}$, so that $M'_{y_{1}}$,  
$M'_{y_{2}}$ can be glued by identifying those components.
 In this way, we obtain a Levi-flat hypersurface $N$ on $S$ tangent to $\mathcal{F}_{S}$. 
\par By doing additional blowing-ups if necessary, we can suppose that $\tilde{\mathcal{F}}$ has reduced singularities. Since $\mathcal{F}_{S}$ is non-dicritical, all irreducible components of $D$ are $\tilde{\mathcal{F}}$-invariants. Let $\tilde{N}$ be the strict transform of $N$ under $\sigma$, then $\tilde{N}\supset D$. In particular, $\tilde{N}$ contains all singularities of $\tilde{\mathcal{F}}$ in $D$. It follows from Proposition \ref{proposition-reducido} that all singularities of $\tilde{\mathcal{F}}$ are saddle with first integral.
\end{proof}
\begin{figure}
\begin{center}
\includegraphics[scale=0.75]{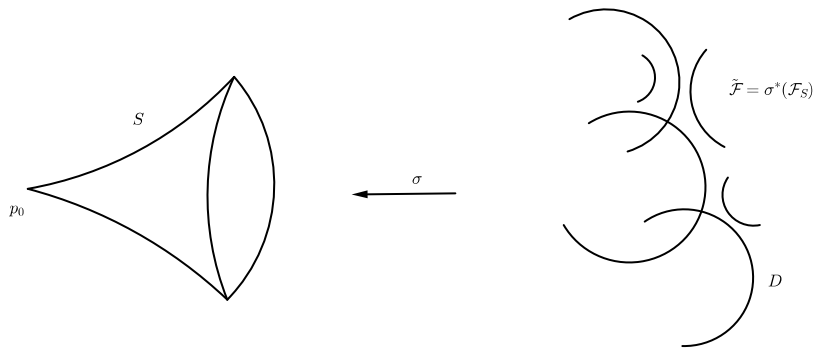}
\caption{Resolution of singularities of $S$ at $p_{0}$.}
\end{center}
\end{figure}
\subsection{End of the proof of Theorem \ref{main-the} in dimension two}
\par The idea is to prove that $\mathcal{F}_{S}$ has a holomorphic first integral.
Since $D$ is invariant by $\tilde{\mathcal{F}}$, i.e., it is the union of leaves and singularities of $\tilde{\mathcal{F}}$, we have $S:=D\backslash \sing(\tilde{\mathcal{F}})$ is a leaf of $\tilde{\mathcal{F}}$. Now, fix $p\in S$
and a transverse section $\sum$ through $p$. By Lemma \ref{sing-int}, the singularities of $\tilde{\mathcal{F}}$ in $D$ are saddle with first integral types. Therefore the transverse section $\sum$ is complete, (see \cite{Loray}, pg. 162).  
Let $G\subset \dif(\sum,p)$ be the Holonomy group of the leaf $S$ of $\tilde{\mathcal{F}}$. It follows from Lemma \ref{cerveau-lema} that all leaves of $\mathcal{F}_{S}$ through points of $N_{reg}$ are closed in $N_{reg}$. This implies that all transformations of $G$ have finite order and $G$ is linearizable. According to \cite{mattei}, $\mathcal{F}_{S}$ has a non-constant holomorphic first integral. Finally from Lemma \ref{web-integral}, $\mathcal{W}$ has a first integral 
as follows: $$P(z)=f_{0}(x,y)+z.f_{1}(x,y)+\ldots+z^{k-1}.f_{k-1}(x,y)+
z^{k},$$ where $f_{0},f_{1},\ldots,f_{k-1}\in\mathcal{O}_{2}$.
 
\subsection{Proof in dimension $n\geq 3$}
Let us give an idea of the proof. First of all, we will prove that there is a holomorphic embedding $i:(\mathbb{C}^{2},0)\rightarrow(\mathbb{C}^{n},0)$ with the following properties:
\begin{enumerate}
\item[(i)] $i^{-1}(M)$ has real codimension one on $(\mathbb{C}^{2},0)$.
\item[(ii)] $i^{*}(\mathcal{W})$ is a $k$-web on $(\mathbb{C}^{2},0)$ tangent to $i^{-1}(M)$.
 \end{enumerate} 
Set $E:=i(\mathbb{C}^{2},0)$. The above conditions and Theorem \ref{main-the} in dimension two imply
that $\mathcal{W}|_{E}$ has a non-constant holomorphic first integral, say $g=f_{0}+z.f_{1}+\ldots+z^{k-1}.f_{k-1}+z^{k}$, where $f_{0},\ldots,f_{k-1}\in\mathcal{O}_{2}$. After that we will use a lemma to prove that 
$g$ can be extended to a holomorphic germ $g_{1}$, which is a first integral of $\mathcal{W}$.
\par Let $\mathcal{F}$ be a germ at $0\in\mathbb{C}^{n}$, $n\geq 3$, of a
holomorphic codimension one foliation, tangent to a real analytic
hypersurface $M$. Let us suppose that $\mathcal{F}$ is defined by $\omega=0$, where $\omega$ is a germ at $0\in\mathbb{C}^{n}$ of an integrable holomorphic 1-form with $cod_{\mathbb{C}^{n}}(\sing(\omega))\geq 2$. We say that a holomorphic embedding $i:(\mathbb{C}^{2},0)\rightarrow(\mathbb{C}^{n},0)$ is transverse to $\omega$ if $cod_{\mathbb{C}^{n}}(\sing(\omega))=2$, which means in fact that, as a germ of set, we have $\sing(i^{*}(\omega))=\{0\}$. Note that the definition is independent of the particular germ of holomorphic 1-form which represents $\mathcal{F}$. Therefore, we will say that the embedding $i$ is transverse to $\mathcal{F}$ if it
is transverse to some holomorphic 1-form $\omega$ representing $\mathcal{F}$.
\par We will use the following lemma of \cite{alcides}.
\begin{lemma}\label{levi-trans}
In the above situation, there exists a 2-plane $E\subset\mathbb{C}^{n}$, transverse
to $\mathcal{F}$, such that the germ at $0\in E$ of $M\cap E$ has real codimension one. 
\end{lemma}
\par We say that a embedding $i$ is transverse to $\mathcal{W}$ if it is transverse to all $k$-foliations which defines $\mathcal{W}$. Now, one deduces the following
\begin{lemma}\label{web-trans}
There exists a 2-plane $E\subset\mathbb{C}^{n}$, transverse
to $\mathcal{W}$, such that the germ at $0\in E$ of $M\cap E$ has real codimension one.
\end{lemma}
\begin{proof}
First of all, note that outside of the discriminant set of $\mathcal{W}$, we can suppose that 
$\mathcal{W}=\mathcal{F}_{1}\boxtimes\ldots\boxtimes\mathcal{F}_{k}$, where $\mathcal{F}_{1},\ldots,\mathcal{F}_{k}$ are germs of codimension one smooth foliations. Since $\mathcal{W}$ is tangent to $M$, there is a foliation $\mathcal{F}_{j}$ such that is tangent to a Levi foliation $\mathcal{L}_{M}$ on $M_{reg}$. Lemma \ref{levi-trans} implies that we can find a 2-plane $E_{0}$ transverse to $M$ and to $\mathcal{F}_{j}$.
Clearly the set of linear mappings transverse to $\mathcal{F}_{1},\ldots,\mathcal{F}_{k}$ simultaneously is open and dense in the set of linear mappings from $\mathbb{C}^{2}$ to $\mathbb{C}^{n}$, by Transversality
theory, there exists a linear embedding $i$ such that $E=i(\mathbb{C}^{2},0)$ is transverse
to $M_{reg}$ and to $\mathcal{W}$ simultaneously. 
\end{proof}

\par Let $E$ be a 2-plane as in Lemma \ref{web-trans}. It easy to check that $\mathcal{W}|_{E}$ satisfies the hypotheses of Theorem \ref{main-the}. By the two dimensional case $\mathcal{W}|_{E}$ has a non-constant first integral:
\begin{equation}
g_{0}+z.g_{1}+\ldots+z^{k-1}.g_{k-1}+z^{k},
 \end{equation}
where $g_{0},\ldots,g_{k-1}\in\mathcal{O}_{2}$.

Let $X$ be the variety associated to $\mathcal{W}$ and set $S$ be the surface associated to $\mathcal{W}|_{E}$. Observe that $\mathcal{F}_{S}$ has a non-constant holomorphic first integral $g$ defined on $S$.
\begin{lemma}
In the above situation, we have $\mathcal{F}_{X}|_{S}=\mathcal{F}_{S}$ and  $\mathcal{F}_{X}$ has a non-constant holomorphic first integral $g_{1}$ on X, such that $g_{1}|_{S}=g$.
\end{lemma}
\begin{proof}

\par It is easily seen that $S\subset X$ which implies that $\mathcal{F}_{X}|_{S}=\mathcal{F}_{S}$. Let us extend $g$ to $X$. Fix $p\in X_{reg}\backslash \sing(\mathcal{F}_{X})$. It is possible to find a small neighborhood $W_{p}\subset X$ of $p$ and a holomorphic coordinate chart $\varphi:W_{p}\rightarrow \triangle$, where $\triangle\subset\mathbb{C}^{n}$ is a polydisc, such that:
\begin{enumerate} 
\item[(i)] $\varphi(S\cap W_{p})=\{z_{3}=\ldots=z_{n}=0\}\cap\triangle$.
\item[(ii)]$\varphi_{*}(\mathcal{F}_{X})$ is given by $dz_{n}|_{\triangle}=0$.

\end{enumerate}
Let $\pi_{n}:\mathbb{C}^{n}\rightarrow\mathbb{C}^{2}$ be the projection defined by $\pi_{n}(z_{1},\ldots,z_{n})=(z_{1},z_{2})$ and set $\tilde{g}_{p}:=g\circ\varphi^{-1}\circ\pi_{n}|_{\triangle}$.  We obtain that $\tilde{g}$ is a holomorphic function defined in $\triangle$ and is a first integral of $\varphi_{*}(\mathcal{F}_{X})$. Let $g_{p}=\tilde{g}_{p}\circ\varphi$. Notice that, if $W_{p}\cap W_{q}\neq \emptyset$, $p$ and $q$ being regular points for $\mathcal{F}_{X}$, then we have $g_{p}|_{W_{p}\cap W_{q}}=g_{q}|_{W_{p}\cap W_{q}}$. This follows  easily form the identity principle for holomorphic functions. In particular, $g$ can be extended to $$W=\bigcup_{p\in X_{reg}\backslash \sing(\mathcal{F}_{X})}W_{p},$$ which is a neighborhood of $X_{reg}\backslash \sing(\mathcal{F}_{X})$. Call $g_{W}$ this extension.
 \par Since $cod_{X_{reg}}\sing(\mathcal{F}_{X})\geq 2$, by a theorem of Levi (cf. \cite{siu}), $g_{W}$ can be extended to $X_{reg}$, as $cod_{X_{reg}}(\sing(X))\geq 2$ this allows us to extend $g_{W}$  to $g_{1}$ as holomorphic first integral for $\mathcal{F}_{X}$, in whole $X$.
\end{proof}
\subsection{End of the proof of Theorem \ref{main-the} in dimension $n\geq 3$}
 Since $\mathcal{F}_{X}$ has a non-constant holomorphic first integral on $X$, Lemma \ref{web-integral} imply that $\mathcal{W}$ has a non-constant holomorphic first integral. 

\vspace*{0.5cm}
\noindent {\textit {Acknowledgments}}.-- 
The author is greatly indebted to Alcides Lins Neto for suggesting the problem and for many stimulating conversations. This is part of the author's Ph.D. thesis, written at IMPA. I also want to thank the referee for his suggestions.

\vskip .2in
\begin{flushleft}
Arturo Fern\'andez-P\'erez\\
{\small Departamento de Matem\'atica, UFMG}\\
{\small Av. Ant\^onio Carlos, 6627 C.P. 702,}\\
{\small 30123-970 - Belo Horizonte - MG, BRAZIL.}\\
{\small E-mail: afernan@impa.br}
\end{flushleft}

\end{document}